\documentclass[11pt]{amsart}
\usepackage{mathrsfs}
\usepackage{amsfonts}
\usepackage{latexsym,amsmath,amssymb}
\renewcommand{\epsilon}{\varepsilon}

\newtheorem{theorem}{Theorem}[section]
\newtheorem{lemma}{Lemma}[section]

\newtheorem{proposition}{Proposition}[section]
\newtheorem{deff}{Definition}[section]

\newcommand{\bth}{\begin{theorem}}
\newcommand{\ble}{\begin{lemma}}
\newcommand{\bcor}{\begin{corr}}
\newcommand{\bdeff}{\begin{deff}}
\newcommand{\bprop}{\begin{proposition}}
 \newcommand{\ele}{\end{lemma}}
 \newcommand{\ecor}{\end{corr}}
 \newcommand{\edeff}{\end{deff}}
 
 \newcommand{\eprop}{\end{proposition}}

 \renewcommand{\Pi}{\varPi}

 \renewcommand{\epsilon}{\varepsilon}

\numberwithin{equation}{section}

\newtheorem{lem}{Lemma}[section]
\newtheorem{cor}[lem]{Corollary}

\title
{On the rigidity theorems for
Lagrangian translating solitons in pseudo-Euclidean space II}
\author{Rongli Huang}
\author{Ruiwei Xu}

\address{School of Mathematics and Statistics, Guangxi Normal University,
Guilin, Guangxi 541004, People's Republic of China,
E-mail:ronglihuangmath@gxnu.edu.cn}
\address{College of Mathematics and Information Science,
Henan Normal University,
Xinxiang, Henan 453007, People's Republic of China,
E-mail:rwxu$@$henannu.edu.cn }
\date{}

\begin{document}
\maketitle

\begin{abstract}
Let $u$ be a smooth convex
function in  $\mathbb{R}^{n}$ and the graph $M_{\nabla u}$ of
$\nabla u$ be a space-like translating soliton in pseudo-Euclidean
space $\mathbb{R}^{2n}_{n}$ with a translating vector
$\frac{1}{n}(a_{1}, a_{2}, \cdots, a_{n}; b_{1}, b_{2}, \cdots,
b_{n})$, then the function $u$ satisfies
$$
\det D^{2}u=\exp \left\{ \sum_{i=1}^n- a_i\frac{\partial u}{\partial
x_{i}} +\sum_{i=1}^n b_ix_i+c\right\} \qquad \hbox{on}\qquad\mathbb
R^n$$ where $a_i$, $b_i$ and $c$ are constants. The Bernstein type results are obtained in the course of the arguments.
 \end{abstract}
\let\thefootnote\relax

\section{Introduction}
Consider the logarithmic Monge-Amp\`{e}re  flow, (cf.\cite{KT})
\begin{equation}\label{e1.6}
\left\{ \begin{aligned}\frac{\partial u}{\partial t}-\frac{1}{n}\ln
\det D^{2}u&=0,
& t>0,\quad x\in \mathbb{R}^{n}, \\
 u&=u_{0}(x), & t=0,\quad x\in \mathbb{R}^{n}.
\end{aligned} \right.
\end{equation}
By Proposition 2.1 in \cite{HB}, there exists a family of diffeomorphisms
$$ r_{t}: \mathbb{R}^{n}\rightarrow \mathbb{R}^{n},$$
such that the map
\begin{equation*}\begin{aligned}
F(x,t)&=(r_{t}(x),  Du(r_{t}(x),t)) \subset
\mathbb{R}^{2n}_{n},\\
F_{0}(x)&=(x, Du_{0}(x))
\end{aligned}
\end{equation*}
satisfies the mean curvature flow in pseudo-Euclidean
space:
\begin{equation*}\label{e1.7}
\left\{ \begin{aligned}\frac{dF}{dt}&=\overrightarrow{H}, \\
F(x,0)&=F_{0}(x),
\end{aligned} \right. \end{equation*}
where $\overrightarrow{H}$ is the mean curvature vector of the
sub-manifold defined by $F$.

\begin{deff}
Assume  that  $u_{0}(x)\in C^{2}(\mathbb{R}^{n})$. We call
$u_{0}(x)$ satisfying condition $\circledS$, if
\begin{equation*}
\Lambda I\geq D^{2}u_{0}(x)\geq \lambda I,\qquad x\in
\mathbb{R}^{n}.
\end{equation*}
Here  $ \Lambda, \lambda$ are two positive constants
 and $I$ is the identity matrix.
\end{deff}
The first author established the long time existence result of
the logarithmic Monge-Amp\`{e}re  flow \cite{HB}.

\begin{proposition}\label{p3.1}
Let $u_{0}:\mathbb{R}^{n}\rightarrow \mathbb{R}$ be a $C^{2}$
function which satisfies  condition $\circledS$.
Then  there exists a unique strictly
convex solution  of (\ref{e1.6}) such that
\begin{equation*}\label{e3.1}
u(x,t)\in C^{\infty}(\mathbb{R}^{n}\times (0,+\infty))\cap
C(\mathbb{R}^{n}\times [0,+\infty))
\end{equation*}
where $u(\cdot,t)$ satisfies condition $\circledS$.  More generally, for
 $l\in\{3,4,5\cdots\}$ and $\epsilon_{0}>0$,  there holds
\begin{equation*}\label{e3.3a}
\sup_{x\in\mathbb{R}^{n}}|D^{l}u(x,t)|^{2}\leq C, \qquad \forall
t\in (\epsilon_{0},+\infty),
\end{equation*}
where $C$  depends only on $n, \lambda, \Lambda, \dfrac{1}{\epsilon_{0}}.$
\end{proposition}

More generally, the following decay estimates were derived
in \cite{HW}.
\begin{proposition}\label{t1.2a}
Assume that   $u(x,t)$ is a strictly convex  solution of
(\ref{e1.6}), and $u(\cdot,t)$ satisfies  condition $\circledS$.
Then there exists a positive constant $C$  depending only on $n,
\lambda, \Lambda, \dfrac{1}{\epsilon_{0}}$, such that
for all $l\in\{3,4,5\cdots\}$ there holds
\begin{equation}\label{e1.7a}
\sup_{x\in\mathbb{R}^{n}}|D^{l}u(x,t)|^{2}\leq
\frac{C}{t^{l-2}}, \qquad \forall t\geq\epsilon_{0}.
\end{equation}
\end{proposition}
Self-shrinking solutions
of the mean curvature flow are determined by the following quasi-linear
elliptic systems
\begin{equation}\label{e1.11}
\vec{H}=-\frac{X^N}{2}.
\end{equation}
In the ambient Euclidean space, the self-shrinkers has been
considered in \cite{ACY},\cite{K}, \cite{QM1}, \cite{QM2}, \cite{IC},
\cite{LH}. In the ambient pseudo-Euclidean space, the self-shrinking
graphs with high codimensions  can be seen in  \cite{ACY}, \cite{HW}, \cite{DW},
\cite{DX}. The solutions are hoped to give a better understanding of
the flow at  type I singularities  by Huisken¡¯s monotonicity
formula. Let $M=\{(x,Du(x))|x\in \mathbb{R}^{n}\}$ be a space-like
submanifold satisfying (\ref{e1.11}) in $\mathbb{R}^{2n}_{n}$ with
the induced metric $ u_{ij} dx_idx_j$.
Then up to an additive constant the
function $u$ is a solution to the Monge-Amp\`{e}re  type
equation
\begin{equation}\label{e1.12}
\det D^{2}u=\exp\left\{n(-u+\frac{1}{2}\sum_{i=1}^{n}x_{i}\frac{\partial
u}{\partial x_{i}})\right\}.
\end{equation}
Q. Ding and Y.L. Xin \cite{DX}  proved that every classical strictly
convex entire solutions of the equation (\ref{e1.12}) must be a quadratic
polynomial.

Another important examples of Type II singularities is a class of
eternal solutions known as translating solitons. From \cite{AT}, we
see that A. Neves and G. Tian gave examples that exclude the
existence of nontrivial translating solutions to Lagrangian mean
curvature flow. Some interesting translating solitons were found by
D. Joyce, Y.I. Lee and M.P. Tsui \cite{DYM} with oscillation of the
Lagrangian angle arbitrarily small. Recently, Mart\'{i}n, Savas-Halilaj and Smoczyk \cite{MSS} obtained classification results and topological obstructions for the existence of translating solitons of the mean curvature flow in Euclidean space. In this paper we will classify the translating solutions of Lagrangian mean curvature flow under
certain convexity assumptions on the generating potential as flat
Lagrangian planes in  pseudo-Euclidean space $\mathbb{R}^{2n}_{n}$.

 A.M. Li and the second author \cite{LX} showed that every
smooth strictly convex solutions  of the
Monge-Amp\`{e}re  type equation
\begin{equation*}\label{e1.2}
\det D^{2}u=\exp\left\{-\sum_{i=1}^{n}d_{i}\frac{\partial u}{\partial
x_{i}}-d_{0}\right\},\,\,\, x\in \mathbb{R}^{n}
\end{equation*}
must be a quadratic polynomial where $d_{0}, d_{1},\cdots,d_{n}$ are
constants. 
Here we consider the following more general Monge-Amp\`{e}re type equation
\begin{equation}\label{e1.3} \det D^{2}u=\exp \left\{
\sum_{i=1}^n- a_i\frac{\partial
u}{\partial x_{i}} +\sum_{i=1}^n b_ix_i+c\right\} \qquad
\hbox{on}\qquad\mathbb R^n,
\end{equation}
where $a_i$, $b_i$ and $c$ are constants. According to the arguements in \cite{CM},  the entire solution to
(\ref{e1.3})  is  a  space-like translating soliton to Lagrangian
mean curvature flow in pseudo-Euclidean space. As for the authors,
it seems that the approach in \cite{LL} can't be applied to here for obtaining the rigidity result
 of the solutions. So we want to  search for new ideas
to prove the theorems similar to J\"{o}rgens \cite{J}, Calabi
\cite{C}, and Pogorelov \cite{P}.

Let $A$ be an $n\times n$ real matrix and define
$$\Lambda=\left\{l\;|\;A^l=I, A^{l-1}\neq I, l\in Z^{+}\right\}, \quad  l_A=\min_{l\in\Lambda} l.$$
Denote $a=(a_{1}, a_{2}, \cdots, a_{n})$, $b=(b_{1}, b_{2}, \cdots,
b_{n})$ and $\langle a,b\rangle$ is the inner product of two
vectors in $\mathbb{R}^n$. As an application of Proposition  \ref{t1.2a}, we can prove
that
\begin{theorem}\label{t1.1}
Let $u$ be a smooth strictly convex solution of (\ref{e1.3}) $(n\geq2)$
where $|a|\neq 0$, $|b|\neq 0$. Suppose that there exists an orthogonal
matrix $A$ such that $l_A\geq 3$ and $u(Ax)=u(x)$ for  each
$x\in\mathbb{R}^{n}$. If the  smallest eigenvalue $\mu(x)$ of
$D^{2}u$ satisfies
\begin{equation}\label{e1.511}
\liminf_{x\rightarrow\infty}|x|\mu(x)> \dfrac{n-1}{|a|\cos\dfrac{\pi}{l_A}},
\end{equation}
then  $u(x)$ must be  a quadratic polynomial.
\end{theorem}
By the methods in the proof of Theorem \ref{t1.1}, we formulate the above result in a more general form when the dimension $n=1$:
\begin{cor}\label{c1.1}
Suppose that $u=u(t)$ satisfies
\begin{equation}\label{e1.4} u''=\exp (- a_0
u' +b_0t+c) \qquad
\hbox{on}\qquad\mathbb R,
\end{equation}
where $a_{0}b_0>0$ and there exists $t_{0}$ such that $u(t-t_{0})=u(t_{0}-t)$
for each $t\in \mathbb R$. Then
\begin{equation*}\label{e1.5}
 u=\frac{b_0}{2a_0}(t-t_{0})^{2}+\min_{\mathbb R}u.
\end{equation*}
\end{cor}
If the potential function $u$ has more symmetry, it is easy to get
\begin{theorem}\label{t1.2}
Let $u$ be a smooth strictly convex radially symmetric solution of (\ref{e1.3}),
then  $u(x)$ must be a quadratic polynomial.
\end{theorem}

We outline our proof as follows. In section 2, we provide preliminary results which will be used in the proof of Theorem 1.1. The techniques used in this section are reflective of those in \cite{HW}, but the corresponding  prior estimates to the solutions in the current scenario need modification because the structure of (\ref{e1.3})  is unlike the self-shrinking equation (\ref{e1.12}).  In section 3, we give the proofs of the main results.

\section{Preliminaries}
Straightforward computation gives the relations of
(\ref{e1.6}) and (\ref{e1.3}).
\begin{lemma}\label{l2.1}
If $u$ is a smooth strictly convex solution of the PDE (\ref{e1.3}) and
define $\tilde{u}(x,t)=u(x-at)+(\langle b,x\rangle-
\frac{1}{2}\langle b,a\rangle t+c)t$. Then $\tilde{u}(x,t)$
satisfies the logarithmic  Monge-Amp\`{e}re flow
\begin{equation*}
\left\{ \begin{aligned}\frac{\partial \tilde{u}}{\partial t}-\ln
\det D^{2}\tilde{u}&=0,
& t>0,\quad x\in \mathbb{R}^{n}, \\
 \tilde{u}&=u(x), & t=0,\quad x\in \mathbb{R}^{n}.
\end{aligned} \right.
\end{equation*}
\end{lemma}

An important consequence of the decay estimates (\ref{e1.7a}) is the following result.
\begin{lemma}\label{l2.2}
If $u$ is a smooth strictly convex solution of the PDE (\ref{e1.3}) and satisfies
condition $\circledS$. Then $u(x)$ must be a quadratic polynomial.
\end{lemma}
\begin{proof}
By Lemma \ref{l2.1} and Proposition \ref{t1.2a}, we have
\begin{equation*}
\sup_{x\in\mathbb{R}^{n}}|D^{3}\tilde{u}(x,t)|^{2}\leq
\frac{C}{t}, \qquad \forall t\geq \epsilon_{0}.
\end{equation*}
That is
\begin{equation*}
\sup_{x\in\mathbb{R}^{n}}|D^{3}u(x-at)|^{2}\leq
\frac{C}{t}, \qquad \forall t\geq \epsilon_{0}.
\end{equation*}
Therefore
\begin{equation*}
\sup_{x\in\mathbb{R}^{n}}|D^{3}u(x)|^{2}\leq
\frac{C}{t}, \qquad \forall t\geq \epsilon_{0}.
\end{equation*}
Let $t\rightarrow+\infty$ then we obtain $D^{3}u\equiv0$
and the claim follows.
\end{proof}
To obtain the first rigidity theorem we search that which
condition can imply condition $\circledS$. Denote $B_{R}$  be a ball centered at $0$ with
radius $R$ in $\mathbb{R}^{n}$.
 \begin{lemma}\label{l1.3}
Let $u :\mathbb{R}^{n}\rightarrow \mathbb{R}$ be a  smooth
strictly convex solution to (\ref{e1.3}) and  $|a|\neq 0$.
Suppose that  there exists an orthogonal matrix $A$
such that $l_A\geq 3$ and $u(Ax)=u(x)$ for  each
$x\in\mathbb{R}^{n}$. If the  smallest eigenvalue $\mu(x)$
of $D^{2}u$ satisfies $(\ref{e1.511})$.
Then there exists a positive constant $R_{0}$ such that
\begin{equation}\label{e2.1}
D^{2}u(x)\leq C I,\qquad x\in
\mathbb{R}^{n},
\end{equation}
where $C$ is a positive constant depending only
on $|a|$, $l_A$, $\mu(x)$ and $\|u\|_{C^{2}(\bar{B}_{R_{0}+1})}$.
\end{lemma}
\begin{proof}
Denote $$u_{i}=\frac{\partial u}{\partial x_{i}},\,\,\, u_{ij}
=\frac{\partial^{2}u}{\partial x_{i}\partial x_{j}},\,\,\,
u_{ijk}=\frac{\partial^{3}u}{\partial x_{i}\partial x_{j}
\partial x_{k}}, \cdots $$ and
$$[u^{ij}]=[ u_{ij}]^{-1},\quad L=u^{ij}
\frac{\partial^{2}}{\partial x_{i}\partial x_{j}}.$$
Let $\gamma$ denote a unite vector field. Set
$$u_{\gamma}=D_{\gamma}u,\,\,\,u_{\gamma\gamma}
=D^{2}_{\gamma\gamma}u.$$
We will prove that
$$\sup_{x\in \mathbb{R}^{n},\,\gamma\in
\mathbb{S}^{n-1}}u_{\gamma\gamma}\leq C.$$
By (\ref{e1.511}), there are some constant
$\lambda> \dfrac{n-1}{|a|\cos\dfrac{\pi}{l_A}}$  and
$R_{0}$, such that
$$|x|\mu(x)\geq \lambda,$$
for $|x|>R_{0}+1$. One can define a family of smooth functions by
$$f_{k}(t)=\left\{ \begin{aligned}
&1 , & 0\leq t\leq R_{0}, \\
& \varphi &R_{0}\leq t\leq R_{0}+1,\\
&-k [t^{2}-(R_{0}+1)^{2}]+\frac{3}{4}, & t\geq R_{0}+1,
\end{aligned} \right.$$
where $0<k\leq 1$,  and $(t,\varphi(t))$ is a smooth curve
connecting two points $(R_{0},1)$, $(R_{0}+1,\dfrac{3}{4})$
satisfying $\dfrac{3}{4}\leq \varphi\leq 1$.

We view $u_{\gamma\gamma}$ as a function on
$\mathbb{R}^{n}\times\mathbb{S}^{n-1}$. It is easy to see that
$f_{k}(|x|)u_{\gamma\gamma}$ always attains its maximum at
$$(p,\xi)\in \{(x,\gamma)\in \mathbb{R}^{n}\times
\mathbb{S}^{n-1}|f_{k}(|x|)>0\}.$$
Using the definition of $l_A$ we can choose the maximum
point $x$, denoted by $p$, such that
\begin{equation}\label{e2.2}
\langle x,-a\rangle\geq |a||x|\cos\frac{\pi}{l_A}.
\end{equation}
By (\ref{e1.511}), we have  $u_{\gamma\gamma}>0$.  Let
  $$\eta_{k}(x)=f_{k}(|x|),\,\,\, w=\eta_{k}(x)u_{\xi\xi}.$$
Then at $p$,
\begin{equation}\label{e2.3}
0\geq
Lw=u^{ij}(\eta_{k}u_{\xi\xi})_{ij}=u^{ij}(\eta_{k})_{ij}
u_{\xi\xi}+2u^{ij}(\eta_{k})_{i}(u_{\xi\xi})_{j}
+\eta_{k}u^{ij}(u_{\xi\xi})_{ij}.
\end{equation}
We assume that
$$p\in\{x\in \mathbb{R}^n||x|> R_{0}+1\}.$$
By a rotation, we can assume
that $D^{2}u$  is diagonal at $p$ with $\xi$ as the $x_{1}$ direction. In
this case, $u_{\xi\xi}=u_{11}$. Then at $p$, there holds
\begin{equation*}
(\eta_{k}u_{11})_{j}=0, \,\,\,\,\, j=1,2,\cdots,n.
\end{equation*}
Hence
\begin{equation}\label{e2.4}
(u_{11})_{j}=-u_{11}\frac{(\eta_{k})_{j}}{\eta_{k}},
\,\,\,(\eta_{k})_{j}=-\eta_{k}\frac{(u_{11})_{j}}{u_{11}},
\,\,\,\,\, j=1,2,\cdots,n.
\end{equation}
Clearly, by (\ref{e2.4}),
\begin{equation}\aligned\label{e2.5}
2u^{ij}(\eta_{k})_{i}(u_{11})_{j}=\; &u^{11}(\eta_{k})_{1}
u_{111}+u^{11}(\eta_{k})_{1}u_{111}+2\sum_{i\neq 1}
\frac{(\eta_{k})_{i}u_{11i}}{u_{ii}}\\
=&-u^{11}\frac{(\eta_{k})_{1}(\eta_{k})_{1}}{\eta_{k}}
u_{11}-u^{11}\eta_{k}\frac{u^{2}_{111}}{u_{11}}-
2\sum_{i\neq 1}\eta_{k}\frac{u^{2}_{11i}}{u_{ii}u_{11}}.
\endaligned
\end{equation}
Differentiating  the equation (\ref{e1.3}), we have
\begin{equation*}
u^{ij}u_{ij1}=-a_{i}u_{i1}+b_{1},
\end{equation*}
\begin{equation}\label{e2.6}
u^{ij}u_{11ij}=\sum_{i,j=1}^{n}\frac{u^{2}_{ij1}}
{u_{ii}u_{jj}}-a_{i}u_{i11}.
\end{equation}
Substituting (\ref{e2.5}), (\ref{e2.6}) into (\ref{e2.3}) and using
$$(\eta_{k})_{i}=-2k x_{i},\qquad
 (\eta_{k})_{ij}=-2k \delta_{ij},$$
we obtain, at $p$,
\begin{equation*}\aligned
0\geq &
- 2k \sum_{i=1}^{n}u^{ii}u_{11}
-\frac{(\eta_{k})^{2}_{1}}{\eta_{k}}-
\eta_{k}\frac{u^{2}_{111}}{u^{2}_{11}}-
2\eta_{k}\sum_{i\neq 1}\frac{u^{2}_{11i}}{u_{ii}u_{11}}\\
&+\eta_{k}\sum_{i,j=1}^{n}\frac{u^{2}_{ij1}}{u_{ii}u_{jj}}
-\eta_{k}a_{i}u_{i11}.
\endaligned
\end{equation*}
Note that
\begin{equation*}
\eta_{k}\sum_{i,j=1}^{n}\frac{u^{2}_{ij1}}{u_{ii}u_{jj}}\geq
\eta_{k}\frac{u^{2}_{111}}{u^{2}_{11}} +2\eta_{k}\sum_{i\neq
1}\frac{u^{2}_{11i}}{u_{ii}u_{11}}.
\end{equation*}
Combining the above two inequalities,  we get
\begin{equation*}
0\geq
- 2k \sum_{i=1}^{n}u^{ii}
u_{11}-\frac{(\eta_{k})^{2}_{1}}{\eta_{k}}
-\eta_{k}a_{i}u_{i11}.
\end{equation*}
In view of (\ref{e2.4}),
\begin{equation*}
\eta_{k}u_{i11}=-u_{11}(\eta_{k})_{i}=2k x_{i}
 u_{11}.
\end{equation*}
Then at $p$,
\begin{equation*}
\frac{(\eta_{k})^{2}_{1}}{\eta_{k}}\geq  - 2k
\sum_{i=1}^{n}u^{ii}u_{11}
-2k a_{i}x_{i}u_{11}.
\end{equation*}
Using (\ref{e2.2}) and $u_{ii}\geq \dfrac{\lambda}{|x|}$
for $i\geq 2$, it follows from the above arguments  that
\begin{equation*}\aligned
\frac{4k^{2}x^{2}_{1}}{\eta_{k} }\geq
 - 2k -\frac{2k(n-1)}
{\lambda }|x|u_{11}
 +2k|a||x|u_{11}\cos\frac{\pi}{l_A}.
\endaligned
\end{equation*}
Noting $\lambda>  \dfrac{n-1 }{|a|\cos\dfrac{\pi}{l_A}}$,
at $p$, we have
\begin{equation*}
\dfrac{1}{\left(|a|\cos\dfrac{\pi}{l_A}-\dfrac{n-1}{\lambda}\right)|x|}
\left(\eta_{k}+ 2kx^{2}_{1} \right)
\geq  \eta_{k}u_{11}.
\end{equation*}
Thus   if $p\in\{x\in\mathbb{R}^{n}||x|> R_{0}+1\}$,  then there holds
\begin{equation}\label{e2.7}
\max_{x\in\mathbb{R}^{n},\gamma\in
\mathbb{S}^{n-1}}\eta_{k}u_{\gamma\gamma}\leq
\dfrac{\dfrac{3}{4}+2k(R_{0}+1)^{2}}{\left(|a|\cos\dfrac{\pi}
{l_A}-\dfrac{n-1}{\lambda}\right)(R_{0}+1)}
.
\end{equation}
 And  if $p\in\{x\in\mathbb{R}^{n}||x|\leq R_{0}+1\}$, then
\begin{equation}\label{e2.8}
\max_{x\in\mathbb{R}^{n},\gamma\in
\mathbb{S}^{n-1}}\eta_{k}u_{\gamma\gamma}\leq
\|u\|_{C^{2}(\bar{B}_{R_{0}+1})}.
\end{equation}
From (\ref{e2.7}) and (\ref{e2.8}), by $k\leq 1$ we obtain
\begin{equation*}\label{e2.9}
\max_{x\in\mathbb{R}^{n},\gamma\in
\mathbb{S}^{n-1}}\eta_{k}u_{\gamma\gamma}\leq
\dfrac{11(R_{0}+1)}{4\left(|a|\cos\dfrac{\pi}{l_A}-
\dfrac{n-1}{\lambda}\right)} +\|u\|_{C^{2}(\bar{B}_{R_{0}+1})}.
\end{equation*}
For any fixed $x\in \mathbb{R}^{n}$ and $\gamma\in
\mathbb{S}^{n-1}$, let $k$ converges to $0$,  then
$$\dfrac{3}{4}u_{\gamma\gamma}\leq \dfrac{11(R_{0}+1)}
{4(|a|\cos\dfrac{\pi}{l_A}-\dfrac{n-1}{\lambda})}
+\|u\|_{C^{2}(\bar{B}_{R_{0}+1})}.$$
So we obtain
$$u_{\gamma\gamma}\leq \dfrac{11(R_{0}+1)}{3\left(|a|
\cos\dfrac{\pi}{l_A}-\dfrac{n-1}{\lambda}\right)}
+\dfrac{4}{3}\|u\|_{C^{2}(\bar{B}_{R_{0}+1})}$$
and inequality (\ref{e2.1}) is proved.
\end{proof}

\section{proof of the main results}

{\bf Proof of Theorem \ref{t1.1}:}

Introduce  Legendre transformation of $u$:
\begin{equation*}
\tilde{x}_{i}=\frac{\partial u}{\partial x_{i}}
,\,\,i=1,2,\cdots,n,\,\,\,u^{*}(\tilde{x}_{1},\cdots,
\tilde{x}_{n}):=\sum_{i=1}^{n}x_{i}\frac{\partial u}
{\partial x_{i}}-u(x),\,\,x\in R^n.
\end{equation*}
In terms of $\tilde{x}_{1},\cdots,\tilde{x}_{n},
u^{*}(\tilde{x}_{1},\cdots,\tilde{x}_{n})$, one can easily check that
$$\left(\frac{\partial^{2} u^{*}}{\partial \tilde{x}_{i}\partial
\tilde{x}_{j}}\right)=\left(\frac{\partial^{2} u}{\partial x_{i}
\partial x_{j}}\right)^{-1}.$$
Thus, in view of (\ref{e2.1}), $$D^{2}u^{*}\geq \frac{1}{C}I.$$ By
the PDE (\ref{e1.3}) we obtain
\begin{equation*}
\det D^{2}u^{*}=\exp\{a_{i}\tilde{x}_{i}-b_{i}u^{*}_{i}-c\}.
\end{equation*}
Since $u(Ax)=u(x)$ one can verify that
$u^{*}(A\tilde{x})=u^{*}(\tilde{x})$ for  each
$\tilde{x}\in\mathbb{R}^{n}$. Using Lemma \ref{l1.3}, we have
 $$D^{2}u^{*}\leq CI.$$
So
$$\frac{1}{C}I\leq D^{2}u\leq CI.$$
An application of  Lemma \ref{l2.2} yields the desired result.
\qed

{\bf Proof of Corollary \ref{c1.1}:}

Without loss of generality, we assume that $t_0=0$.
The equation (\ref{e1.4}) shows that $u$ must be strictly
convex. For $n=1$, similar to  the proof of Lemma \ref{l1.3},
the inequality (\ref{e2.2})
can be replaced by
\begin{equation*}
\langle t,-a_0\rangle\geq |a_{0}||t|.
\end{equation*}
Here we use the symmetry condition. Following the procedure in the proof of Theorem 1.1,
we obtain estimates (\ref{e2.1}) for $n=1$ and  then we  arrive at the conclusion of Corollary \ref{c1.1}.\qed

{\bf Proof of Theorem \ref{t1.2}:}
  Now we assume that $u$ is radially symmetric function, then $$\Psi=\ln\det(u_{ij})=\ln u_{rr}+(n-1)(\ln u_r-\ln r) $$ is also radially symmetric and depends only on $|x|$. Similar to the arguments in \cite{ACY},
  it follows that $\ln\det(u_{ij})$ must then attain either alocal maximum or a local minimum over any open ball $B$ in $\mathbb{R}^n$. From (\ref{e1.3}), we have
\begin{equation}\label{e2.10} u^{ij}\Psi_{ij}-\langle a,D\Psi\rangle=0.\end{equation}
Applying the strong maximum principle to (\ref {e2.10}), we see that $\Psi$ is constant in $B$, and hence in $\mathbb{R}^n$. From the classical Jorgens-Calabi-Pogorelov theorem, we complete
the proof of theorem. \qed

\vspace{5mm}
{\bf Acknowledgment:} The first author was supported by NNSF of China (Grant
No.11261008) and NNSF of Guangxi (Grant No.2012GXNSFBA053009) and
 was very grateful to  Institute of Differential Geometry at Leibniz University
Hannover for the kind hospitality.
The second author was partially supported by
NSFC (Grant No.11101129).

\end{document}